\documentclass[reqno,12pt,dvipsnames ]{amsart}

\usepackage[a4paper,margin=1.1 in]{geometry}
\usepackage{amsmath,amssymb,amsthm,mathtools}
\usepackage{enumitem}
\usepackage{microtype}

\usepackage[colorlinks=true,
            linkcolor=red,
            citecolor=blue,
            urlcolor=blue]{hyperref}

\numberwithin{equation}{section}
\newtheorem{theorem}{Theorem}[section]
\newtheorem{proposition}[theorem]{Proposition}

\newtheorem{claim}[theorem]{Claim}
\newtheorem{corollary}[theorem]{Corollary}
\theoremstyle{definition}

\theoremstyle{remark}
\newtheorem{remark}[theorem]{Remark}
\newtheorem{example}[theorem]{Example}
\newcommand{\C}{\mathbb C}

\newcommand{\ord}{\operatorname{ord}}
\newcommand{\hot}{\operatorname{h.o.t}}
\setlist[enumerate]{itemsep=3pt,topsep=5pt}
\title[Bi-Lipschitz invariance of the transverse polynomial]{Bi-Lipschitz invariance of the transverse polynomial along polar arcs}
\author{Nhan Nguyen}
\address{FPT University, Danang, Vietnam}
\email{nguyenxuanvietnhan@gmail.com}



\begin{document}
\maketitle

\begin{abstract}
We study transverse polynomials along polar arcs of reduced
holomorphic plane function germs. We prove that, under
bi-Lipschitz right equivalence, matched tangential polar arcs
have the same transverse order and their transverse polynomials
agree up to nonzero rescalings of the source and target,
at every rational scale $1<q<d(\gamma)$, where $d(\gamma)$
is the gradient canyon degree. This range is sharp as the result can fail at both endpoints, even under analytic right equivalence.

As an application, we recover the bi-Lipschitz invariance of the augmented Newton polygon proved by Migus--P\u{a}unescu--Tib\u{a}r. We also give an example showing that the transverse polynomial contains new information not detected by the polar-value invariants and augmented Newton polygons.
\end{abstract}

\section{Introduction}
Unlike the Lipschitz geometry of complex set germs, bi-Lipschitz right equivalence of analytic function germs admits continuous moduli, as proved by Henry and Parusi\'nski \cite{HP03, HP04}. For holomorphic plane function germs, their invariant is obtained from the leading terms of the function along tangential polar arcs. 

A geometric framework for these phenomena was subsequently developed by
P\u{a}unescu--Tib\u{a}r \cite{PT19} through the theory of gradient canyons.  A polar arc
$\gamma$ determines a gradient degree $d(\gamma)$ and a corresponding canyon,
which describes the smallest transverse scale at which the order of the gradient stabilizes. It was shown in \cite{PT19} that  gradient canyons, their degrees, multiplicities and mutual contacts are preserved by bi-Lipschitz right equivalence.  In particular, the leading Henry--Parusi\'nski modulus can be recovered through this canyon structure \cite{MPT22}.

More recently, finer invariants have been extracted from the values of the
function along different polar arcs.  In \cite{Ng24}, we introduced an
invariant, denoted $\rm{Inv}^2$, obtained by comparing normalized
polar-value expansions.  An equivalent formulation was independently obtained by Migus--P\u{a}unescu--Tib\u{a}r \cite{MPT25}, who also developed a recursive construction of higher-level invariants. In \cite{MPT26}, they proved the bi-Lipschitz invariance of the augmented Newton polygon associated with $f_x$ along a gradient canyon.

 In this paper, we study transverse polynomials along polar arcs.
Let $f:(\C^2,0)\to(\C,0)$ be a reduced holomorphic
germ, and let $\gamma$ be a tangential polar arc of canyon degree
$d>1$. For each rational $q\geq1$, we write 
$$f(\gamma(y)+zy^q,y)-f(\gamma(y),y) = y^{H_f(q)}(P_{f,\gamma, q}(z) + o(1));$$ and call the first nonzero coefficient  polynomial 
$P_{f,\gamma,q}$ the transverse polynomial at
scale $q$.

This leading polynomial construction was used by
Koike--Parusi\'nski \cite{KP13} for arbitrary allowable arcs
in the real analytic setting; here we apply the same construction
to polar arcs. They proved that bi-Lipschitz right equivalences
carrying one allowable arc onto another as set germs preserve
the corresponding relative Newton boundaries
\cite[Corollary~1.6]{KP13}. If the equivalence is $C^1$ and
tangent to the identity, the associated polynomials also
coincide \cite[Proposition~1.7]{KP13}. More generally, this equality holds if the image of one
allowable arc is contained in a suitable small horn neighbourhood of
the other  \cite[Proposition~1.9]{KP13}.

Our main result (Theorem \ref{thm_main}) shows that bi-Lipschitz right equivalence preserves the transverse orders and transverse polynomials along matched
tangential polar arcs, up to nonzero rescalings, at every rational
scale $1<q<d$. More precisely, if $f \sim_{\text{bi-Lip}} g$ and $\gamma$
and $\beta$ are matched polar arcs with common canyon degree
$d>1$, then
$$
H_f(q)=H_g(q)=H(q),
\qquad
P_{f,\gamma,q}(z)=A^{H(q)}P_{g,\beta,q}(B_qz),
$$
where $A,B_q\in\C^*$. The constant $A$ is determined by
the longitudinal scaling in the canyon correspondence and is
independent of $q$, while $B_q$ may depend on $q$. The range $1<q<d$ is sharp: the conclusion may fail at either
endpoint, even under analytic right equivalence
(see Examples~\ref{ex_1} and~\ref{ex_2}).
The proof uses the correspondence of gradient canyons by Migus-P\u{a}unescu--Tib\u{a}r \cite{MPT25},
rescaling around points of a polar arc, and the affine
rigidity result of C\^amara--Fernandes \cite[Proposition 3.1]{CF25}.

Section~\ref{section2.1} studies how the transverse polynomial
depends on the scale. The function $H_f(q)$ is piecewise affine,
with breakpoints at the contact orders
$\delta_j=\ord(\gamma_j-\gamma)$.
Away from these scales, the transverse polynomial is a monomial.
At $q=\delta_j$, the coefficient $a_j$ in
$$
\gamma_j(y)-\gamma(y)=a_jy^{\delta_j}+\hot
$$
appears as a critical point of the transverse polynomial. 

The main theorem is presented in Section \ref{section2.2}. Section~\ref{section3} relates $H_f(q)-q$ to the supporting lines
of the augmented Newton polygon, giving a short proof of the
invariance theorem of Migus--P\u{a}unescu--Tib\u{a}r
\cite{MPT26}. Finally, Section~\ref{section4} gives an example showing that the transverse polynomial provides information
not detected by the polar-value invariants and augmented Newton polygons discussed above.

\subsection*{AI Assistance}
ChatGPT helped in designing the example in Section \ref{section4}. We take full responsibility for its correctness.

\subsection*{Acknowledgments}
 The research was supported by the Vietnam Ministry of Education and Training (MOET) under grant number B2025-CTT-01.

\section{Polar arcs and transverse polynomials}\label{section2}

Let $f:(\C^2,0)\to(\C,0)$ be a reduced holomorphic function germ of multiplicity
$m\ge2$. Choose coordinates in which $f$ is mini-regular in $x$, that is,
$f_m(1,0)\ne0$, where $f_m$ is its initial homogeneous part.
A \emph{polar arc} of $f$ is a convergent Puiseux arc $(\gamma(y),y)$ satisfying
$$
 f_x(\gamma(y),y)=0,\quad f(\gamma(y),y)\not\equiv0.
$$
We identify this arc with its Puiseux series $\gamma$ and put
$h=\ord_y f(\gamma(y),y)$.  We regard distinct Puiseux conjugates as distinct polar arcs.

The \emph{canyon degree} $d=d(\gamma)$ is the smallest rational number $q\ge1$
such that, for generic $z\in\C$,
\begin{equation}\label{eq_degree}
 \ord_y\|\nabla f(\gamma(y)+zy^q,y)\|
 =\ord_y\|\nabla f(\gamma(y),y)\|.
\end{equation}
The \emph{gradient canyon} of $\gamma$ consists of the Puiseux arcs
$\alpha(y)=\gamma(y)+O(y^d)$.

A polar arc is called \emph{tangential} if its tangent line is contained in
the tangent cone of $V(f) :=\{f=0\}$. Thus, if $\gamma(y)=a y+o(y)$, it is tangential
exactly when $f_m(a,1)=0$.
For such an arc, $h>m$. Since $\gamma$ is polar, the chain rule gives
$$
 \frac{d}{dy}f(\gamma(y),y)=f_y(\gamma(y),y),\quad
 \ord_y\|\nabla f(\gamma(y),y)\|=h-1>m-1.
$$
On the other hand, the gradient along $(\gamma(y)+zy,y)$ has order $m-1$
for generic $z$. Hence, \eqref{eq_degree} does not hold at $q=1$, and $d(\gamma)>1.$

From now on, $\gamma$ is tangential. After the linear change of coordinates $x\mapsto x+ay$,
we may assume $\gamma(y)=o(y)$, and hence $\gamma'(y)=o(1)$. This coordinate change preserves
the polar equation and mini-regularity.

For a rational number $q\ge1$, define $H_f(q)$ and the nonzero polynomial $P_{f,\gamma,q}\in\C[z]$ by
\begin{equation}\label{eq_transverse}
 f(\gamma(y)+zy^q,y)-f(\gamma(y),y)
 =y^{H_f(q)}\bigl(P_{f,\gamma,q}(z)+o(1)\bigr), 
\end{equation}
where $o(1)\to0$ as $y\to0$ locally uniformly in $z\in\C$. We call $P_{f,\gamma,q}$ the \emph{transverse polynomial at scale $q$}.

Let us go a bit further.  For fixed $y$, around  $x = \gamma(y)$, 
\begin{equation}
     f(\gamma(y) + X, y) - f(\gamma(y), y) =  f_x (\gamma(y), y) X + \frac{1}{2} f_{xx}(\gamma(y), y) X^2 + \cdots
\end{equation}
Since $\gamma$ is a polar arc, $f_x (\gamma(y), y) = 0$. Thus, 

\begin{equation}\label{eq_taylor}
f(\gamma(y)+X,y)
 =a(y)+\sum_{j\ge2}a_j(y)X^j,
\end{equation}
where $a(y):= f(\gamma(y),y) \text{ and } \ord_y a=h$.
If
$a_j\not\equiv0$, write
$$
 a_j(y)=c_jy^{\nu_j}+\hot,\quad c_j\ne0.
$$
Substituting $X=zy^q$ into \eqref{eq_taylor} gives
\begin{equation}\label{eq_minimum}
 H_f(q)=\min_{j\geq 2}(\nu_j+jq),\quad
 P_{f,\gamma,q}(z)=\sum_{\nu_j+jq=H_f(q)}c_jz^j.
\end{equation}
In particular, $\ord_z(P_{f,\gamma,q})\geq 2$. 

Now, by differentiating in $X$, we have
$$
 \ord_y f_x(\gamma(y)+zy^q,y)=H_f(q)-q
 \quad\text{for generic }z.
$$
Differentiating 
\eqref{eq_taylor} in $y$ gives 
\begin{equation}
 f_y(\gamma(y)+X,y)
 =a'(y)+\sum_{j\ge2}a'_j(y)X^j
       -\gamma'(y)\sum_{j\ge2}j a_j(y)X^{j-1}.
\end{equation}
For $|X|\le C|y|^q$, the two last sums on the right have sizes
$O(|y|^{H_f(q)-1})$ and $O(|y|^{H_f(q)-q})$, respectively. Since $q>1$ and $\gamma'(y)=o(1)$, we obtain
\begin{equation}\label{eq_fy-expansion}
 f_y(\gamma(y)+X,y)
 =a'(y)+o(|y|^{H_f(q)-q}).
\end{equation}
Together with the generic order of $f_x$, we have 
\begin{equation}
 \ord_y\|\nabla f(\gamma(y)+zy^q,y)\|
 =\min\{h-1,H_f(q)-q\}
 \quad\text{for generic }z.
\end{equation}
Now set 
$$ L_f(q):= H_f(q)-q=\min_j(\nu_j+(j-1)q).$$ 
It is obvious that $L_f(q)$ is continuous and strictly
increasing.  From definition of canyon degree,  $L_f(d)\ge h-1$ where $d:= d(\gamma)$. Continuity and strict monotonicity imply
\begin{equation}\label{eq_L(d)}
    L_f(q)<h-1\quad(1<q<d) \quad\text{and} \quad  L_f(d)=h-1.
\end{equation}
Equivalently,
$$
 H_f(q)<h+q-1\quad(1<q<d) \quad\text{and} \quad
 H_f(d)=h+d-1.
$$
These explicit formulas for $H_f(q)$ and $L_f(q)$ will be used later.

\subsection{How the transverse polynomial changes with the scale}\label{section2.1}
Consider the Puiseux factorization: 
\begin{equation}\label{eq_Puiseux-fac}
    f_x(x, y) = u(x, y) \prod_{j =  0}^r (x - \gamma_j(y))^{m_j}, \quad u(0, 0)\neq 0,
\end{equation}
where 
$\gamma_0,\ldots,\gamma_r$ are polar arcs. Fix a polar arc $\gamma = \gamma_i$. Without loss of generality, we may assume that  $i  = 0$. For $j \geq 1$,  write
$$
\gamma_j(y)-\gamma(y)
=
a_jy^{\delta_j}+\text{h.o.t.},
\quad
a_j\neq 0,
$$
where
$\delta_j=\operatorname{ord}_y(\gamma_j-\gamma)
$.

\begin{proposition}\label{prop_variation_scale}
Let $q>1$ be a rational number. Then
$$
P'_{f,\gamma,q}(z)
=
C_q z^{M_q}
\prod_{\delta_j=q}(z-a_j)^{m_j},
\quad
C_q\neq0,
$$
where
$
M_q
=
m_0+\sum_{\delta_j>q}m_j
$.

Consequently, if $q$ is not equal to any contact order $\delta_j$, then
$$
P_{f,\gamma,q}(z)
=
c_q z^{M_q+1},
\quad
c_q\neq0.
$$
\end{proposition}

\begin{proof}
Translate the first variable in \eqref{eq_Puiseux-fac} by putting $
x=\gamma(y)+X$. We have
$$
f_x(\gamma(y)+X,y)
=
U(X,y)
X^{m_0}
\prod_{j=1}^{r}
\bigl(X-(\gamma_j(y)-\gamma(y))\bigr)^{m_j},
$$
where $U(X,y)
=u(\gamma(y)+X,y).$ Now put $X=zy^q$. For each $j\geq1$, there are three possible cases. 

If $\delta_j>q$, then
$$
zy^q-(\gamma_j(y)-\gamma(y))
=
y^q\bigl(z+o(1)\bigr).
$$
Hence, this branch contributes the factor $z^{m_j}$ to the leading
polynomial in $z$.

If $\delta_j=q$, then
$$
zy^q-(\gamma_j(y)-\gamma(y))
=
y^q\bigl(z-a_j+o(1)\bigr).
$$
This branch contributes the factor $(z-a_j)^{m_j}$.

Finally, if $\delta_j<q$, then
$$
zy^q-(\gamma_j(y)-\gamma(y))
=
-a_jy^{\delta_j}(1+o(1)).
$$
This branch contributes only a nonzero constant. 

Therefore, 
$$
f_x(\gamma(y)+zy^q,y)
=
y^{L(q)}
\left(
C_q
z^{M_q}
\prod_{\delta_j=q}(z-a_j)^{m_j}
+
o(1)
\right),
$$
where $C_q\neq0$, $
M_q :=m_0+\sum_{\delta_j>q}m_j, L(q):=qm_0
+\sum_{j=1}^{N} m_j\min\{q,\delta_j\}
$.

By definition, 
$$
f(\gamma(y)+zy^q,y)-f(\gamma(y),y)
=
y^{H_f(q)}
\bigl(P_{f,\gamma,q}(z)+o(1)\bigr).
$$
Differentiating with respect to $z$ gives
$$
y^q f_x(\gamma(y)+zy^q,y)
=
y^{H_f(q)}
\bigl(P'_{f,\gamma,q}(z)+o(1)\bigr).
$$
Equivalently,
$$
f_x(\gamma(y)+zy^q,y)
=
y^{H_f(q)-q}
\bigl(P'_{f,\gamma,q}(z)+o(1)\bigr).
$$
This yields
$$
H_f(q)-q=L(q) \quad \text{ and } P'_{f,\gamma,q}(z)
=
C_q
z^{M_q}
\prod_{\delta_j=q}(z-a_j)^{m_j}.
$$
If $q\neq \delta_j$, then we obtain
$P'_{f,\gamma,q}(z)=C_qz^{M_q}$. Since $P_{f,\gamma,q}(0)=0$, 
integration gives
$P_{f,\gamma,q}(z)=\frac{C_q}{M_q+1}z^{M_q+1}$.
\end{proof}

\begin{remark}\label{rem_2} It follows directly from Proposition \ref{prop_variation_scale} that  on every open interval containing no contact order, $P_{f, \gamma, q}(z)$ is a single monomial. At a contact scale $q=\delta_j$, the corresponding polar branch appears at finite position $z=a_j$, and
$P'_{f,\gamma,q}$ contains the factor $(z-a_j)^{m_j}$. Thus, the transverse polynomial  $P_{f, \gamma, q}(z)$ changes only when $q$ reaches
the contact order of another polar branch.
\end{remark}

\subsection{Bi-Lipschitz invariance of transverse polynomial}\label{section2.2}

We will use the following correspondence of gradient canyons \cite[Lemmas~3.1--3.2]{MPT25}, based on \cite[Theorem~5.8]{PT19}.

\begin{proposition}\label{prop_matching}
Let $f=g\circ\varphi$, where $\varphi=(\varphi_1,\varphi_2)$ is a bi-Lipschitz
homeomorphism germ, and let $\gamma$ be a tangential polar arc of $f$ of canyon
degree $d>1$. There is a polar arc $\beta$ of $g$ in the corresponding canyon
with the same canyon degree such that
\begin{equation}\label{eq_matching}
 \varphi(\gamma(y),y)
 =\bigl(\beta(Y(y))+O(|Y(y)|^d),\,Y(y)\bigr),
 \quad Y(y)=Ay+o(y),\quad A\in\C^*.
\end{equation}
Moreover,
$$
 \ord_y f(\gamma(y),y)=\ord_Y g(\beta(Y),Y)=h.
$$
The coordinates may be chosen so that both germs are mini-regular in their first variable and $\gamma(y)=o(y)$, $\beta(Y)=o(Y)$.
\end{proposition}
We call the polar arc $\beta$  in Proposition \ref{prop_matching} \textit{the matching polar} of $\gamma$.  Our main result is as follows. 
\begin{theorem}[Transformation law] \label{thm_main}
Let $f=g\circ\varphi$, and let $\gamma,\beta$ be matched polar arcs as in Proposition 
\ref{prop_matching}. For every rational $q$ with $1<q<d$,
\begin{equation}\label{eq:exponent-invariance}
 H_f(q)=H_g(q)=:H(q).
\end{equation}
For the same constant $A$ as in \eqref{eq_matching}, there exists
$B_q\in\C^*$ such that
\begin{equation}\label{eq_transformation_law}
 P_{f,\gamma,q}(z)=A^{H(q)}P_{g,\beta,q}(B_qz)
 \quad\text{for all }z\in\C.
\end{equation}
Here, $A$ is independent of $q$ while $B_q$ may depend on $q$.

In particular, by the continuity, $H_f(q) = H_g(q)$ for all $1\leq q \leq d$. 
\end{theorem}

\begin{proof}

Fix $1<q<d$ and write $H_f=H_f(q)$ and $H_g=H_g(q)$.

If
$$
 |x-\gamma(y)|\le C|y|^q,\quad |v-y|\le C|y|^q,
$$
then, by
integrating \eqref{eq_fy-expansion}  we have 
\begin{equation}\label{eq_longitudinal}
 f(x,v)-f(x,y)
 =O(|y|^{h+q-1})+o(|y|^{H_f})
 =o(|y|^{H_f}).
\end{equation}
The corresponding estimate also holds for $g$.

For $t>0$ tending to $0$, we define
\begin{equation}\label{eq_rescaling}
 \Phi_t(z,w)=
 \frac{\varphi(\gamma(t)+t^qz,t+t^qw)-\varphi(\gamma(t),t)}{t^q}.
\end{equation}
The maps $\Phi_t$ fix the origin and have a common bi-Lipschitz constant.
Their domains, and those of their inverses, contain every fixed compact
subset of $\C^2$ once $t$ is sufficiently small.  Arzel\`a--Ascoli theorem, applied
to the maps and their inverses, 
gives a sequence $t_n \to 0$ and  a global bi-Lipschitz homeomorphism
$\Phi:\C^2\to\C^2$ such that $\Phi_{t_n}\to\Phi$ locally uniformly and
$\Phi(0)=0$.

For the target, Proposition \ref{prop_matching} gives 
$$
 \varphi(\gamma(t),t)=(\beta(Y_t)+e_t,Y_t),\quad
 Y_t/t\longrightarrow A,\quad e_t=O(t^d).
$$

The estimate \eqref{eq_longitudinal} together with the defining transverse
expansion gives
\begin{equation}\label{eq_source-limit}
 \frac{f(\gamma(t)+t^qz,t+t^qw)-f(\gamma(t),t)}{t^{H_f}}
 \longrightarrow P_{f,\gamma,q}(z)
\end{equation}
locally uniformly in $(z,w)$.

Because $q<d$, we have $e_t=o(t^q)$. Then, 
\[
 \begin{split}
 &g(\beta(Y_t)+e_t+t^qu,Y_t)-g(\beta(Y_t)+e_t,Y_t)\\
 &\quad=Y_t^{H_g}\left[
 P_{g,\beta,q}\!\left(\frac{e_t+t^qu}{Y_t^q}\right)
 -P_{g,\beta,q}\!\left(\frac{e_t}{Y_t^q}\right)+o(1)\right].
 \end{split}
\]

Since $e_t/Y_t^q\to0$ and $P_{g,\beta,q}(0)=0$, we have
\begin{equation}\label{eq_target-limit}
 \begin{split}
 &\frac{g(\beta(Y_t)+e_t+t^qu,Y_t+t^qv)
             -g(\beta(Y_t)+e_t,Y_t)}{t^{H_g}}\\
 &\hspace{25mm}\longrightarrow
 A^{H_g}P_{g,\beta,q}(A^{-q}u),
 \end{split}
\end{equation}
locally uniformly in $(u,v)$.

Let
$$
 p_t=(\gamma(t),t),\quad
 p_t(z,w)=(\gamma(t)+t^qz,t+t^qw).
$$
Writing $\Phi_t=(\Phi_{t,1},\Phi_{t,2})$, the definition \eqref{eq_rescaling} gives
$$
 \varphi(p_t(z,w))
 =
 \bigl(
 \beta(Y_t)+e_t+t^q\Phi_{t,1}(z,w),
 Y_t+t^q\Phi_{t,2}(z,w)
 \bigr).
$$
Since $f=g\circ\varphi$, we have 
$$
 f(p_t(z,w))-f(p_t)
 =
 g(\varphi(p_t(z,w)))-g(\varphi(p_t)).
$$
Combining \eqref{eq_source-limit} and \eqref{eq_target-limit} (replace $(u, v)$ by $(\Phi_{t,1}, \Phi_{t,2})$) yields

\begin{equation}\label{eq_comparison}
    t^{H_f}\bigl(P_{f,\gamma,q}(z)+o(1)\bigr)
 =
 t^{H_g}\left(
 A^{H_g}P_{g,\beta,q}
 \bigl(A^{-q}\Phi_{t,1}(z,w)\bigr)+o(1)
 \right),
\end{equation}
where the remainders are locally uniform in $(z,w)$.

If $H_f<H_g$, division by $t^{H_f}$ and passing the limit would
gives $P_{f,\gamma,q}\equiv0$. If $H_f>H_g$, division by $t^{H_g}$ and taking the limit as $t_n\to 0$ gives
$$
 P_{g,\beta,q}(A^{-q}\Phi_1(z,w))=0
 \quad\text{for all }(z,w)\in\C^2.
$$
But $\Phi$ is onto, so its first component takes every value in $\C$;
this would force $P_{g,\beta,q}\equiv0$. Both are impossible.
Therefore $H_f=H_g=:H(q)$, and \eqref{eq_comparison} becomes
$$
 P_{f,\gamma,q}(z)
 =A^{H(q)}P_{g,\beta,q}(A^{-q}\Phi_1(z,w)).
$$
Set $U(z):=A^{-q}\Phi_1(z,0)$. This is a globally Lipschitz
map $\C\to\C$ with $U(0)=0$, and
$$
P_{f,\gamma,q}=A^{H(q)}P_{g,\beta,q}\circ U.
$$
By \cite[Proposition~3.1]{CF25}, $U$ must be affine. Thus,
$U(z)=B_qz$ for some $B_q\ne0$. This proves  \eqref{eq_transformation_law}. 
\end{proof}

\begin{corollary}\label{Cor_main}
Under the assumptions of Theorem \ref{thm_main}, suppose that, for some
$1<q<d$,
$$
P_{f,\gamma,q}(z)
=
A_2z^2+A_3z^3+A_4z^4,
\quad
A_2A_4\neq0.
$$
Then,
$$
J(P_{f,\gamma,q})
:=
\frac{A_3^2}{A_2A_4}
$$
is preserved under the matching of polar arcs.
\end{corollary}
\begin{proof}
By Theorem \ref{thm_main}, $
P_{f,\gamma,q}(z)
= a\,P_{g,\beta,q}(bz)$
for some $a,b\in\C^*$. It follows that 
$$
P_{g,\beta,q}(z)=B_2z^2+B_3z^3+B_4z^4,
$$
and 
$$
A_j=ab^jB_j,\qquad j=2,3,4,
$$
Therefore,
$$
\frac{A_3^2}{A_2A_4}
=
\frac{a^2b^6B_3^2}
{(ab^2B_2)(ab^4B_4)}
=
\frac{B_3^2}{B_2B_4}.
$$
\end{proof}

\begin{remark}\label{rem_1} The restriction $1 < q < d$ in Theorem~\ref{thm_main} is sharp: as shown in the examples below, the transformation law fails both at scale $q = 1$ and at scale $q = d$ (even for analytic right equivalence).

\begin{example}\label{ex_1} Consider $f(x, y) = x^2 (x-y) (x-2y) + y^6$, $\phi(x, y) = (x, y-x)$, and $g = f\circ \phi$. Clearly, $f$ and $g$ are analytically right equivalent, and hence bi-Lipschitz right equivalent. 

First, we have
$$ f_x(x, y) = x (4x^2 -9xy + 4y^2),$$
$$g(x, y) = x^2(2x-y)(3x-2y)+(y-x)^6$$
and 
$$g_x(x, y) = x(24x^2 - 21xy + 4y^2) - 6(y-x)^5.$$
Hence, polar arcs of $f$ and $g$ respectively are: $$\gamma_0(y) = 0, \quad \gamma_{1,2}(y) = \frac{9\pm\sqrt{17}}{8}y,$$
and
$$\beta_0(y) = 0 + o(y), \beta_{1,2}(y) = \frac{21\pm \sqrt{57}}{48}y + o(y).$$  

At the scale $q = 1$, 
 $$
f(\gamma_0(y) +zy,y)-f(\gamma_0(y),y)=f(zy,y)-f(0,y)
=
y^4 z^2(z-1)(z-2),
$$
so
$$
P_{f,\gamma_0,1}(z)
= z^2(z-1)(z-2)= z^4-3z^3+2z^2.
$$
Similar computation gives
$$ P_{f,\gamma_1,1}(z) = z^4 + \frac{ 3+\sqrt{17}}{2} z^3 + \frac{17+9\sqrt{17}}{16} z^2$$
and
$$ P_{f,\gamma_2,1}(z) = z^4 + \frac{ 3-\sqrt{17}}{2} z^3 + \frac{17-9\sqrt{17}}{16} z^2,$$
and 
$$ P_{g,\beta_0,1}(z) = 6z^4 - 7z^3 + 2 z^2.$$
It follows that 
$$ J(P_{f,\gamma_0,1}(z)) =\frac{9}{2} , \quad J(P_{f,\gamma_1,1}(z)) = \frac{104 +24\sqrt{17}}{17 +9\sqrt{17}}  , \quad J(P_{f,\gamma_2,1}(z)) = \frac{104 -24\sqrt{17}}{17 -9\sqrt{17}} $$
and 
$$ J(P_{g,\beta_0,1}(z)) = \frac{49}{12}.$$
Since $J(P_{g,\beta_0,1}(z))$ is different from the three $J$-values of $f$, the
transformation rule of Theorem~\ref{thm_main} fails at the scale  $q = 1$. 
\end{example}

\begin{example}\label{ex_2} Consider
$$
f(x,y)=x^4+x^2y^4+y^7
$$
and
$$
g(x,y)=x^4+x^2(y-\frac 67 x)^4+(y-\frac 67 x)^7.
$$
Then
$$
g=f\circ\phi,
\quad 
\phi(x,y)=(x,y-\frac 67 x),
$$
so $f$ and $g$ are analytically right equivalent , and hence
bi-Lipschitz right equivalent.

We first compute the transverse polynomials of all polar arcs of $f$
at the canyon degree. Since
$$
f_x=2x(2x^2+y^4),
$$
the polar arcs are
$$
\gamma_0(y)=0,
\quad
\gamma_i(y)=c_i y^2, 
\quad
c_i^2=-\frac12, i = 1, 2
$$
They all satisfy
$$
\ord_y f(\gamma_i(y),y)=7,
\quad
d(\gamma_i)=2.
$$
Put $R(w)=w^4+w^2$. The centered expansion is
$$
f((c+z)y^2,y)-f(cy^2,y)
=y^8\bigl(R(c+z)-R(c)\bigr).$$
Expanding gives
$$
R(c+z)-R(c)
=z^4+4cz^3+(6c^2+1)z^2+(4c^3+2c)z.
$$
The linear term vanishes because \(c\) is a critical point. Hence
$$
P_{f,\gamma,2}(z)=z^4+4cz^3+(6c^2+1)z^2.
$$
Consequently,
$$
J(P_{f,\gamma,2})
=\frac{16c^2}{6c^2+1}
=
\begin{cases}
0,&c=0,\\[2mm]
4,&c^2=-\frac12.
\end{cases}
$$

For $g$, we have
$$
g_x(x,y)
=
4x^3
+2x\left(y-\frac67x\right)^4
-\frac{24}{7}x^2\left(y-\frac67x\right)^3
-6\left(y-\frac67x\right)^6.
$$
Thus, $g$ has three polar arcs $$\beta_i = a_i y^2 + \hot, \quad i = 0, 1, 2$$
where $a_i$ are roots of the equation $4a^3 + 2a -6 = 0$.
Since $a=1$ is a simple root, there exists a polar arc $\beta_0(y) = y^2 + \hot$. It is easy to check that 
$$\ord_y g(\beta_0, y) = 7, \quad d(\beta_0) = 2.$$
Computation gives
$$ P_{g,\beta_0, 2}(z) = z^4 + 4z^3 + 7z^2$$
Thus, $J( P_{g,\beta_0, 2}(z)) = \frac{16}{7} \not \in \{ 0, 4\}$. By Corollary \ref{Cor_main}, the transformation law fails at the scale
$q=d$.
\end{example}

\end{remark}

\section{ Migus--P\u{a}unescu--Tib\u{a}r's augmented Newton polygon}\label{section3}
In this section, we give a short proof of  Migus--P\u{a}unescu--Tib\u{a}r's result \cite{MPT26} by relating the exponent of the transverse polynomial to the augmented  Newton polygon.  

Let us first recall the definition of augmented Newton polygon.  Let $\gamma$ be a tangential polar arc of $f$.
Denote by
 $ \mathcal{NP}(f_x,\gamma)$ the
 Newton polygon of $f_x(X + \gamma(Y), Y)$.   Let $$
h=h(\gamma)
=
\operatorname{ord}_y f(\gamma(y),y)$$
The \textit{augmented Newton polygon
of $f_x$} along $\gamma$ is
$$
\widehat{\mathcal{NP}}(f_x,\gamma)
:=
\operatorname{Conv}
\left(
\mathcal{NP}(f_x,\gamma)
\cup
\{(0,h-1)\}
\right).
$$

\begin{theorem}[{\cite[Theorem 4.5]{MPT26}}]\label{thm_MPT}
Let $f=g\circ\varphi$, where $\varphi$ is a bi-Lipschitz
homeomorphism germ, and let $\gamma$ and $\beta$ be matched
tangential polar arcs of $f$ and $g$, respectively. 
Then,
$$
\widehat{\mathcal{NP}}(f_x,\gamma)
=
\widehat{\mathcal{NP}}(g_x,\beta)
$$
\end{theorem}

\begin{proof}
Set $
F(X,Y):=f(X+\gamma(Y),Y)$. 
Then $
F_X(X,Y)=f_x(X+\gamma(Y),Y)$. 
As in \eqref{eq_taylor}, write
$$
F(X,Y)
=
F(0,Y)+\sum_{j\geq 2}a_j(Y)X^j,
$$
where $
a_j(Y)=c_jY^{\nu_j}+\hot,  c_j\neq 0$
and
 $\ord_Y F(0,Y)=h$. 

Observe that  
$
F_X(X,Y)
=
\sum_{j\geq 2}j\,a_j(Y)X^{j-1}
$; and  the corresponding leading Newton dots of $F_X$ are
$
(j-1,\nu_j), j\geq 2
$.  For $q\geq 1$, since $
L_f(q) =
 \min_{j\geq 2}\{\nu_j+(j-1)q\}$, $qu+v=L_f(q)$
is the supporting line of $\mathcal{NP}(f_x,\gamma)$ of co-slope $q$.

Since $f$ is mini-regular in $x$, all Puiseux roots of $f_x$ have order at least $1$.  It follows that every compact edge of
$\mathcal{NP}(f_x,\gamma)$ has co-slope at least $1$. On the other hand,  by \eqref{eq_L(d)}, $L_f(d)=h-1$, hence the line $du+v=h-1$ passes through the augmented point $(0,h-1)$ and supports $\mathcal{NP}(f_x,\gamma)$. It is the top supporting line of $\widehat{\mathcal{NP}}(f_x,\gamma)$, and every compact edge of the polygon has co-slope at most $d$. 

In addition, since $L_f(q)$ with $1\leq q\leq d$ is increasing, $L_f(q)\leq L_f(d)=h-1$. The added point $(0,h-1)$ lies on or above the line
$qu+v=L_f(q)$,
so this line is also the supporting line of $\widehat{\mathcal{NP}}(f_x,\gamma)$ for every $1\leq q\leq d$. The same  holds for $\widehat{\mathcal{NP}}(g_x,\beta)$.

By Theorem \ref{thm_main}, for $1\leq q\leq d$, $
H_f(q)=H_g(q)$, so $L_f(q) = L_g(q)$. 
The two augmented Newton polygons have the same supporting line for every co-slope $q\in[1,d]$. Mini-regularity gives the common right endpoint $(m-1,0)$, while augmentation gives the common left endpoint $(0,h-1)$ of the top edge. As a consequence, 
$
\widehat{\mathcal{NP}}(f_x,\gamma)
=
\widehat{\mathcal{NP}}(g_x,\beta)
$.
\end{proof}


\section{Examples}\label{section4}
The following  example shows that the transverse polynomial
contains information which is not detected by the leading Henry--Parusi\'nski data \cite{HP03}, by higher-order invariant \cite{MPT25} or by augmented Newton polygon data \cite{MPT26}.

Consider the polynomials
$$
P(z)=-6z^4+6z
$$
and
$$
Q(z)=2z^4+6z^2-2z+3.
$$
Define the two holomorphic germs
$$
f(x,y)
=
y^{11}+y^{12}P\left(\frac{x}{y^3}\right)
$$
and
$$
g(x,y)
=
y^{11}+y^{12}Q\left(\frac{x}{y^3}\right).
$$
After expansion,
$$
f(x,y)
=
-6x^4+6xy^9+y^{11}
$$
and
$$
g(x,y)
= 2x^4+6x^2y^6 -2xy^9+y^{11}+3y^{12}.
$$

The expressions involving $x/y^3$ are convenient to compute the transverse polynomials at scale $q = 3$. 

\begin{claim}
The polynomials $P$ and $Q$ have three simple critical points and
the same three distinct critical values. These critical values are
the roots of
$$
v^3-\frac{729}{32}=0.
$$
\end{claim}

\begin{proof}
For $P$, we have $
P'(z)=-24z^3+6=6(1-4z^3).$
Hence, the critical points $c_1,c_2,c_3$ are the three roots of
$
c^3=\frac14.
$
At such a critical point,
$$
P(c)
=
-6c^4+6c
=
-6c\left(\frac14\right)+6c
=
\frac92c.
$$
Hence,
$$
P(c)^3
=
\left(\frac92\right)^3 c^3
=
\frac{729}{8}\cdot\frac14
=
\frac{729}{32}.
$$
Thus the critical values of $P$ are exactly the three roots of
$$
v^3-\frac{729}{32}=0.
$$

For $Q$,
$$
Q'(z)=8z^3+12z-2
     =2(4z^3+6z-1).
$$
Hence, its critical points $b_1,b_2,b_3$ are the roots of
$4b^3+6b-1=0$. They are all simple since there is no common root of 
 $4b^3+6b-1$ and its derivative
$12b^2+6$. Modulo the relation $
4b^3+6b-1=0 $ we have 
$$
Q(b)
=
3b^2-\frac32b+3 = \frac 3 2(2b^2 - b + 2).
$$
Expand
$$
(2b^2 - b + 2)^3 = 8b^6 - 12b^5 + 30b^4 - 25b^3 + 30b^2 - 12b + 8.
$$
From $4b^3+6b-1=0$, we get
$$
b^3 = \frac{1}{4} - \frac{3}{2}b,
$$
then
$$
b^4 = \frac{b}{4} - \frac{3}{2}b^2, \quad
b^5 = \frac{1}{4}b^2 - \frac{3}{2}b^3 = \frac{9}{4}b + \frac{1}{4}b^2 - \frac{3}{8},
$$
and
$$
b^6 = (b^3)^2 = \left(\frac{1}{4} - \frac{3}{2}b\right)^2 = \frac{1}{16} - \frac{3}{4}b + \frac{9}{4}b^2.
$$
Substituting these into the expansion gives $
(2b^2 - b + 2)^3 = \frac{27}{4}$. 
Hence,  $
Q(b)^3=\frac{729}{32}$. 

Now we show that these three critical values are distinct. Suppose on the contradiction that two distinct critical points $b_1,b_2$ satisfy $
Q(b_1)=Q(b_2)$.
Since $Q(b)=3b^2-\frac32b+3$, we have  
$$
3(b_1^2-b_2^2)
-\frac32(b_1-b_2)=0. 
$$
It follows that $
b_1+b_2=\frac12$ (since $b_1 \neq b_2)$. 
The sum of the three roots of
$
4b^3+6b-1=0$
is zero, so the third root must be 
$b_3=-\frac12$. 
However, 
$$
4\left(-\frac12\right)^3
+
6\left(-\frac12\right)-1
=
-\frac12-3-1
\neq0,
$$
a contradiction.
\end{proof}

Let $v_1,v_2,v_3$
denote these common critical values. Let $c_i$ and $b_i$ denote the critical points of $P$ and $Q$ respectively such that 
$$
P(c_i)=Q(b_i)=v_i,
\quad i=1,2,3.
$$

\begin{claim}
    Henry--Parusi\'nski polar data are the same
\end{claim}
\begin{proof}
Since $
f_x(x,y)
=
-24x^3+6y^9$, there are three polar arcs 
$$
\gamma_i(y)=c_i y^3,
\quad
c_i^3=\frac14.
$$
These are precisely the critical points of $P$. Similarly,
$$
g_x(x,y)
=
8x^3+12xy^6-2y^9.
$$
Putting $x=by^3$
gives $
g_x(by^3,y)
=
2y^9(4b^3+6b-1)$. Hence, the polar arcs of $g$ are
$$\beta_i(y)=b_i y^3, \text{ where } 4b_i^3+6b_i-1=0.$$ By construction,
\begin{equation}\label{eq_expression_f}
f(\gamma_i(y),y)
=
y^{11}+P(c_i)y^{12}
=
y^{11}+v_i y^{12},
\end{equation}
and
\begin{equation}\label{eq_expression_g}
g(\beta_i(y),y)
=
y^{11}+Q(b_i)y^{12}
=
y^{11}+v_i y^{12}.
\end{equation}
Consequently,
$$
f(\gamma_i(y),y)
=
g(\beta_i(y),y)
=
y^{11}+v_i y^{12}.
$$
Thus, the leading Henry--Parusi\'nski polar data of the
two germs agree.
\end{proof}

\begin{claim}
    Higher invariants agree
\end{claim} 
\begin{proof}
It is obvious that every pair of distinct polar arcs of $f$ and $g$ has contact order $\delta = 3$. All polars have canyon degree $d= 4$. Since the expressions \eqref{eq_expression_f} and \eqref{eq_expression_g} are the same, the complete collections of 
higher-level polar-value data of \cite{MPT25} agree. 
\end{proof}

\begin{claim} Augmented Newton polygons coincide \end{claim}
\begin{proof}
Along a polar arc $
x=c_i y^3+X$, since $
f_x(x,y)
=
y^9
P'\left(\frac{x}{y^3}\right)$
we have
$$
f_x(c_i y^3+X,y)
=
y^9
P'\left(c_i+\frac{X}{y^3}\right).
$$
Since $P'(c_i)=0$,
$$
f_x(c_i y^3+X,y)
=
P''(c_i)y^6X
+
\frac12P'''(c_i)y^3X^2
+
\frac16P^{(4)}(c_i)X^3.
$$
All the corresponding coefficients are nonzero. The  Newton dots are therefore
$$
(1,6),\quad(2,3),\quad(3,0).
$$
The same calculation applies to $g$:
$$
g_x(b_i y^3+X,y)
=
Q''(b_i)y^6X
+
\frac12Q'''(b_i)y^3X^2
+
\frac16Q^{(4)}(b_i)X^3.
$$
Again the Newton dots are
$$
(1,6),\quad(2,3),\quad(3,0).
$$
Since $h-1=10$,
both germs have the same augmented polygon,
with vertices
$$
(0,10),\quad (1,6),\quad (3,0).
$$
\end{proof}

\begin{claim}

The germs $f$ and $g$ are not bi-Lipschitz right equivalent.
\end{claim}

\begin{proof}

The relevant transition occurs at
$q=3< d = 4$. 
For $f$, $$P_{f,\gamma_i,3}(z)=P(c_i+z)-P(c_i).$$
Since $
P(z)=-6z^4+6z$
and
$
c_i^3=\frac14,$
we obtain
$$
\begin{aligned}
P(c_i+z)-P(c_i)
&=
-6\bigl((c_i+z)^4-c_i^4\bigr)+6z\\
&=
-24c_i^3z
-36c_i^2z^2
-24c_i z^3
-6z^4
+6z.
\end{aligned}
$$
Since  $24c_i^3=6$,
$$
P_{f,\gamma_i,3}(z)
=
-36c_i^2z^2
-24c_i z^3
-6z^4.
$$

For $g$,
$
P_{g,\beta_i,3}(z)
=
Q(b_i+z)-Q(b_i)
$.
Since
$$
Q(z)=2z^4+6z^2-2z+3,
$$
we obtain
$$Q(b_i+z)-Q(b_i)
=(8b_i^3+12b_i-2)z + (12b_i^2+6)z^2 +8b_i z^3+2z^4.
$$
Since $
4b_i^3+6b_i-1=0$, 
$$
P_{g,\beta_i,3}(z) = 6(2b_i^2+1)z^2 + 8b_i z^3 + 2z^4.$$
Note from Corollary \ref{Cor_main} that for
$$
R(z)=A_2z^2+A_3z^3+A_4z^4,
\quad A_2A_4\neq0,
$$
the quantity $
J(R) = \frac{A_3^2}{A_2A_4}$
is preserved under the marching polar. 

For $f$, we have
$$
A_2=-36c_i^2,
\quad
A_3=-24c_i,
\quad
A_4=-6.
$$
Thus
$$
J(P_{f,\gamma_i,3}) = \frac{(-24c_i)^2}{(-36c_i^2 (-6)} = \frac{576c_i^2}{216c_i^2} =\frac83
$$

For $g$,
$$
A_2=6(2b_i^2+1),
\quad
A_3=8b_i,
\quad
A_4=2.
$$
Hence
$$ J(P_{g,\beta_i,3})= \frac{(8b_i)^2}{6(2b_i^2+1)\cdot 2}=
\frac{16b_i^2}{3(2b_i^2+1)} \neq  \frac 83
$$

Since  $J(P_{g,\beta_i,3})\neq J(P_{f,\gamma_j,3})$
for every $i,j$, by Corollary  \ref{Cor_main},  $f$ and $g$ are not bi-Lipschitz right equivalent.
\end{proof}

\bibliographystyle{abbrv}
\bibliography{references.bib}
\end{document}